\documentclass[12pt,twopage]{paper}

\usepackage{times}
\usepackage{amssymb}
\usepackage{amsfonts}
\usepackage{amsmath}
\usepackage{bold-extra}
\usepackage{pifont}
\usepackage{yfonts}
\usepackage{url}
\usepackage{hyperref}
\usepackage{fancyhdr}

\newtheorem{theorem}{\textbf{\textsc{Theorem}}}[section]
\newtheorem{definition}[theorem]{\textbf{\textsc{Definition}}}

\newtheorem{convention}[theorem]{\textbf{\textsc{Convention}}}
\newtheorem{lemma}[theorem]{\textbf{\textsc{Lemma}}}

\newtheorem{proposition}[theorem]{\textbf{\textsc{Proposition}}}

\newtheorem{remark}[theorem]{\textbf{\textsc{Remark}}}

\newenvironment{proof}
{\noindent\mbox{\textsf{\textbf{\textsc{Proof}}}:}}
{\hfill{\scriptsize \mbox{\underline{\tt\textbf{QED}\,}$\!|$}}\bigskip}

\title{On the Truth of G\"odelian and Rosserian Sentences}

\author{
{\sc Ziba \    Assadi} \\
\textrm{\large Independent Scholar, Tabriz, IRAN.}\\ \textrm{\large  E-mail:\!~\textsf{\normalsize assadi.golzar@gmail.com}} 
\\
{\sc Saeed \  Salehi  } \\
\textrm{\large Research Institute for Fundamental Sciences \textup{(RIFS)}, University of Tabriz,}\\ \textrm{\large P.O.Box~51666--16471, Tabriz, IRAN.  \, E-mail:\!~\textsf{\normalsize salehipour@tabrizu.ac.ir}} }

\begin{document}

\maketitle

\pagestyle{fancy}
\lhead[ ]{ \thepage\quad{\sc Ziba Assadi \& Saeed Salehi} ({\sf 2020}) }
\chead[ ]{ }
\rhead[ ]{ {\em On the Truth of G\"odelian and Rosserian Sentences} }
\lfoot[ ]{ }
\cfoot[ ]{ }
\rfoot[ ]{ }

\begin{abstract}
There is a longstanding debate in the logico-philosophical community as to why the  G\"odelian sentences of a consistent and sufficiently strong theory are true. The prevalent argument seems to be something like this: since every one of the G\"odelian sentences of such a theory is equivalent to the theory's consistency statement, even provably so inside the theory,  the truth of those sentences follows from the consistency of the theory in question. So,  G\"odelian sentences of consistent theories should be true. In this paper, we show that  G\"odelian sentences of only sound theories are true; and there is a long road from consistency to soundness, indeed a hierarchy of conditions which are satisfied by some theories and  falsified by others.
We also study the truth of Rosserian sentences and provide necessary and sufficient conditions for the truth of Rosserian (and also G\"odelian) sentences of theories.

\noindent
{\bf Keywords}:
The Incompleteness Theorem; G\"odelian Sentences, Rosser's Trick, Rosserian Sentences, Soundness, \\ Consistency, $\Sigma_n$-Soundness.

\noindent
{\bf 2020 AMS MSC}:
03F40.  	
\end{abstract}

\section{Introduction}\label{sec:intro}

By the first incompleteness   theorem
 of G\"odel (1931),  for every consistent and sufficiently strong arithmetical theory there are  sentences which are undecidable in the theory (\cite{Godel31}).
Examples of such undecidable sentences are actually constructed in G\"odel's original proof (hence {\em G\"odelian sentences}) each of which is equivalent to its own unprovability in the theory; see Definition~\ref{def:g} below.
A natural question  here is that while the theory in question cannot decide the truth of its G\"odelian  sentences, what about us (humans)? Can we ``see'' (or demonstrate) their truth?
This question has attracted the attention of many philosophers, physicists, computer scientists, as well as mathematical logicians.
 As there are numerous papers and books on this subject, it is not possible to cite them all here; see the  Conclusions for a few.
In this paper, we present  necessary and sufficient conditions for the truth of  G\"odelian sentences (and  Rosserian sentences) of consistent and   sufficiently strong  arithmetical theories (see the diagram in the Conclusions section).

We assume familiarity with the notions of $\Pi_n$ and $\Sigma_n$ formulae, Robinson's Arithmetic $\textit{\textbf{Q}}$,  and the fact that $\textit{\textbf{Q}}$ is a sound and $\Sigma_1$-complete theory (i.e., every $\textit{\textbf{Q}}$-provable sentence is true and
 every true $\Sigma_1$-sentence is $\textit{\textbf{Q}}$-provable); see e.g. \cite{Kaye91}.
By the Diagonal Lemma of G\"odel and Carnap, see e.g. \cite{Salehi20}, for every formula $\Psi(x)$ with the only free variable $x$, there exists a sentence $\theta$ such that $\theta\!\leftrightarrow\!\Psi(\#\theta)$ is true (in the standard model of natural numbers $\mathbb{N}$) and also provable in $\textit{\textbf{Q}}$; here $\#A$ denotes the numeral of the  G\"odel code of $A$, relative  to a fixed G\"odel numbering (arithmetization) of the syntax.
Moreover, if
$\Psi(x)$ is a $\Pi_n$-formula, then $\theta$ can be taken to be a $\Pi_n$-sentence; and if $\Psi(x)$ is $\Sigma_n$,
 then $\theta$ can be taken to be $\Sigma_n$ too. We provide more details in the following:

\begin{lemma}[The Diagonal Lemma]\label{lem:d}
Let $n\!\geqslant\!1$.

\noindent
For every $\Pi_n$-formula $\Psi(x)$ there exists a $\Pi_n$-sentence $\theta$ such that
$\textit{\textbf{Q}}\vdash \theta\!\leftrightarrow\!\Psi(\#\theta).$

\noindent
And for every $\Sigma_n$-formula $\Psi(x)$ there exists a $\Sigma_n$-sentence $\theta$ with the above property.
\end{lemma}

\begin{proof}
 There exists a $\Sigma_1$-formula $\delta(x,y)$, in the language of arithmetic, that strongly represents the diagonal (primitive recursive) function ${\textsf{d}}$  which assigns to a given $m$ the G\"odel code of the expression that results from substituting $\overline{m}$ (the numeral of $m$, a term in the language of arithmetic representing $m$) for all the free variables (if any) of the expression coded by $m$. So, for every $m\!\in\!\mathbb{N}$ we have $\textit{\textbf{Q}}\vdash\forall y [\delta(\overline{m},y)\!\leftrightarrow\!y\!=\!
 \overline{{\textsf{d}}(m)}]$.

 If $\Psi(x)$ is a $\Pi_n$-formula, then put $\alpha(x)\!=\!\forall y[\delta(x,y)\!\rightarrow\!\Psi(y)]$, and let ${\sf a}$ be its G\"odel code. Now, let $\theta\!=\!\alpha(\overline{{\sf a}})$; then $\theta$ is a $\Pi_n$-sentence and we have provably in $\textit{\textbf{Q}}$ that

\;  $\theta
 \!=\!
 \forall y[\delta(\overline{{\sf a}},y)\!\rightarrow\!\Psi(y)]
 \!\leftrightarrow\!
  \forall y[y\!=\!\overline{{\textsf{d}}({\sf a})}
  \!\rightarrow\!\Psi(y)]
   \!\leftrightarrow\!
   \forall y[y\!=\!\#\theta
  \!\rightarrow\!\Psi(y)]
\!\leftrightarrow\! \Psi(\#\theta)$.

 If $\Psi(x)$ is a $\Sigma_n$-formula, then put $\eta(x)\!=\!\exists y[\delta(x,y)\!\wedge\!\Psi(y)]$, and let ${\sf e}$ be its G\"odel code. Now, let $\theta\!=\!\alpha(\overline{\sf e})$; then $\theta$ is a $\Sigma_n$-sentence and we have provably in $\textit{\textbf{Q}}$ that

\;  $\theta
 \!=\!
 \exists y[\delta(\overline{\sf e},y)\!\wedge\!\Psi(y)]
 \!\leftrightarrow\!
  \exists y[y\!=\!\overline{{\textsf{d}}({\sf e})}
  \!\wedge\!\Psi(y)]
   \!\leftrightarrow\!
   \exists y[y\!=\!\#\theta
  \!\wedge\!\Psi(y)]
\!\leftrightarrow\! \Psi(\#\theta)$.
\end{proof}

This lemma is one of the breakthroughs of  G\"odel's theorem (and of modern logic). The incompleteness theorem is usually stated for recursively enumerable ({\sc re}) theories that extend $\textit{\textbf{Q}}$; though it also holds for more general theories, see e.g. \cite{SalSer17}.
If a theory is {\sc re}, then it can be defined by a $\Sigma_1$-formula, see   \cite[Theorem~3.3.]{Kaye91}, and so its theorems (i.e., provable sentences) can also be  defined by a $\Sigma_1$-formula.

\begin{convention}\label{conv}
Throughout, let $T$ be an {\sc re} extension of  $\textit{\textbf{Q}}$. Let ${\tt Pr}_T(x)$ be a provability predicate for $T$, which is a $\Sigma_1$-formula, relative to a fixed G\"odel numbering.
A basic property of  ${\tt Pr}_T(x)$  is that for every sentence $\varphi$ we have
$$T\vdash\varphi   \; \iff \;    \mathbb{N}\vDash{\tt Pr}_T(\#\varphi) \; \iff \;    \textit{\textbf{Q}}\vdash{\tt Pr}_T(\#\varphi).$$
Let ${\tt Con}_T\!=\!\neg{\tt Pr}_T(\#[0\!\neq\!0])$ be the \textup{($\Pi_1$-)}sentence that expresses $T$'s consistency.
\hfill $\diamond $
\end{convention}

\section{G\"odel Sentences and their Truth.}
G\"odel's proof of his incompleteness theorem uses the diagonal lemma for the negation of the provability predicate of $T$.
\begin{definition}[G\"odelian Sentences]\label{def:g}
A sentence $\gamma$ is called a G\"odelian sentence of $T$ when we have $T\vdash\gamma\!\leftrightarrow\!\neg{\tt Pr}_T(\#\gamma)$.
\hfill $\diamond $
\end{definition}

Since for sufficiently strong theories $T$ any two G\"odelian   sentences are $T$-provably equivalent, see e.g.  \cite[Remark~2.2.5.]{Smorynski77},  many authors talk of {\em the} G\"odel sentence of $T$. However, we will
show that every unsound theory has both true and false G\"odel sentences; so, even though the theory proves (unsoundly) that those sentences are equivalent, in reality they are not.
G\"odel~\cite{Godel31}, and several authors after him, e.g. \cite[p.~825]{Smorynski77}, \cite[p.~7]{Smorynski85}, or \cite[p.~6]{Peregin07}, argue that   G\"odelian  sentences of a sufficiently strong theory are true, since

\qquad (1) they are provably equivalent with their unprovability in the theory, and

\qquad (2) they are indeed unprovable in the theory; and so

\qquad (3) they must be true.

It is argued in~\cite{LajSal19} that
 this line of reasoning does not demonstrate  the truth of G\"odelian sentences, and indeed some ($\Sigma_1$-)unsound theories may have false G\"odelian sentences.
In fact, the assumption (2) in the above argument (of G\"odel) is redundant:

\begin{lemma}\label{lem:ung}
If $T$ is consistent, then for every sentence $\varphi$,  $T\vdash\varphi\!\rightarrow\!\neg{\tt Pr}_T(\#\varphi)$ implies $T\nvdash\varphi$.
\end{lemma}
\begin{proof}
Since $T\vdash\varphi$ would imply on the one hand $T\vdash\neg{\tt Pr}_T(\#\varphi)$ by the assumption, and on the other hand $T\vdash{\tt Pr}_T(\#\varphi)$  by Convention~\ref{conv}.
\end{proof}

So, the question of the validity of the above reasoning for the truth of G\"odelian sentences boils down to the following question:

\medskip

\centerline{\sl Does $T\vdash\gamma\!\leftrightarrow\!
\neg{\tt Pr}_T(\#\gamma)$, for a consistent $T$, imply that $\mathbb{N}\vDash\gamma$?}

\noindent
Or put in another way,
{\sl under which conditions  all the G\"odelian sentences of $T$ are true?}

\medskip

We answer this question in the present section, and in the next section we answer a similar question for the Rosserian sentences (of arithmetical theories).
Let us start with a characterization of the unprovable  sentences:

\begin{proposition}[Characterizing Unprovable  Sentences]\label{prop:g}
Suppose that L\"ob's Rule holds for  $T$.
The following are equivalent for every sentence $\varphi$:
\begin{enumerate}\itemindent=5em
\item[(1)] $\varphi$ is unprovable in $T$, i.e., $T\nvdash\varphi$;
\item[(2)] $\varphi$ is a G\"odelian sentence of
some consistent extension $U$ of $T$;
\item[(3)] $T+[\varphi\!\leftrightarrow\!
\neg{\tt Pr}_T(\#\varphi)]$ is consistent.
\end{enumerate}
\end{proposition}
\begin{proof}

$(1\!\Rightarrow\!2)$:   There exists, by Lemma~\ref{lem:d}, a sentence $\xi$ such that $$T\vdash\xi\!\leftrightarrow\![\varphi
 \!\leftrightarrow\!\neg{\tt Pr}_{T+\xi}(\#\varphi)].$$ Let $U\!=\!T\!+\!\xi$; then $U\vdash\varphi\!\leftrightarrow\!
\neg{\tt Pr}_U(\#\varphi)$ and it remains to show that $U$ is consistent. If not, then $T\vdash\neg\xi$. So, on the one hand we have (i) $T\vdash\neg[\varphi
 \!\leftrightarrow\!\neg{\tt Pr}_{U}(\#\varphi)]$, and on the other hand $U\vdash\varphi$ which implies (ii) $T\vdash{\tt Pr}_{U}(\#\varphi)$ by Convention~\ref{conv}. Now, (i) and (ii) imply that $T\vdash\varphi$, contradicting the assumption.

$(2\!\Rightarrow\!3)$:
If $T+[\varphi\!\leftrightarrow\!
\neg{\tt Pr}_T(\#\varphi)]$ is not consistent, then $T\vdash\neg[\varphi\!\leftrightarrow\!
\neg{\tt Pr}_T(\#\varphi)]$, and so $T\vdash{\tt Pr}_T(\#\varphi)\!\rightarrow\!\varphi$, which implies $T\vdash\varphi$ by L\"ob's Rule. So, for every  extension $U$ of $T$ we have $U\vdash\varphi$, and so, by Convention~\ref{conv},  $U\vdash{\tt Pr}_{U}(\#\varphi)$. Therefore, for every such $U$ we have   $U\vdash\neg[\varphi\!\leftrightarrow\!
\neg{\tt Pr}_U(\#\varphi)]$, which contradicts the assumption.

$(3\!\Rightarrow\!1)$:
 If $T\vdash\varphi$, then, by Convention~\ref{conv}, we have that $T\vdash{\tt Pr}_{T}(\#\varphi)$, and so we should have also
 $T\vdash\neg[\varphi\!\leftrightarrow\!
\neg{\tt Pr}_T(\#\varphi)]$.
  \end{proof}

It should be noted that the assumption of holding L\"ob's Rule for $T$ was used only in the implication $(2\!\Rightarrow\!3)$. So, (1) and (2) are equivalent with each other, and are implied by (3), even when this rule does not hold.
We now provide a necessary and sufficient condition, on a sufficiently strong $T$, for the truth of  all the G\"odelian $\Pi_1$-sentences of $T$:

\begin{theorem}[On the Truth and Independence of G\"odelian $\Pi_1$-Sentences]\label{thm:pi1g}
\,

\noindent
Suppose that $T$ satisfies the following two conditions:

\qquad\quad  \textup{(I)}\; \  $T\vdash\neg{\tt Con}_T\!\rightarrow\!{\tt Pr}_T(\#\varphi)$ for every sentence $\varphi$, and

\qquad\quad  \textup{(II)}\,  $T\vdash{\tt Con}_T\!\rightarrow\!\neg{\tt Pr}_T(\#\gamma)$ for every G\"odelian sentence $\gamma$ of $T$.

If $T\vdash\neg{\tt Con}_T$, then every false $\Pi_1$-sentence is a G\"odelian sentence of $T$, and no G\"odelian sentence of $T$ is independent from $T$.

If $T\nvdash\neg{\tt Con}_T$, then all the G\"odelian $\Pi_1$-sentences of $T$ are true, and all the G\"odelian sentences of $T$ are independent from $T$.
\end{theorem}
\begin{proof}
If $T\vdash\neg{\tt Con}_T$, then by (I) we have $T\vdash{\tt Pr}_T(\#\varphi)$ for every sentence $\varphi$. So, for every G\"odelian sentence $\gamma$ of $T$ we have $T\vdash\neg\gamma$; thus no G\"odelian sentence of $T$ can be independent from $T$. Now, let $\phi$ be an arbitrary false $\Pi_1$-sentence; then $\neg\phi$ is a true $\Sigma_1$-sentence, and so provable in $\textit{\textbf{Q}}$. Thus, $T\vdash\neg\phi$;  and so  from $T\vdash{\tt Pr}_T(\#\phi)$ we have $T\vdash\phi\!\leftrightarrow\!\neg{\tt Pr}_T(\#\phi)$, which means that $\phi$ is a (false) G\"odelian $\Pi_1$-sentence of $T$.

If $T\nvdash\neg{\tt Con}_T$, then by (II)
for every G\"odelian sentence $\gamma$ of $T$
we have $T\nvdash{\tt Pr}_T(\#\gamma)$, and so  $T\nvdash\neg\gamma$; thus $\gamma$ is independent from $T$ (noting that $T$ is consistent and so we also have $T\nvdash\gamma$ by Lemma~\ref{lem:ung}).
If a G\"odelian $\Pi_1$-sentence $\tau$ of $T$
 is not true, then $\neg\tau$ is a true $\Sigma_1$-sentence, and so should be $\textit{\textbf{Q}}$-provable; a contradiction with the $T$-independence of $\tau$ proved above.
\end{proof}

Every extension of Peano's Arithmetic $\textit{\textbf{PA}}$ satisfies
 (I) and (II) in Theorem~\ref{thm:pi1g}; as a matter of fact  (II) is a formalization of G\"odel's first incompleteness theorem (in $T$).
 If $T$ is ($\Sigma_1$-)sound, then
$T\nvdash\neg{\tt Con}_T$. If $T$ is inconsistent or
$T\!=\!S\!+\!\neg{\tt Con}_S$, where $S$ is a consistent extension of $\textit{\textbf{PA}}$, then
$T\vdash\neg{\tt Con}_T$; in the latter case $T$ is  consistent by G\"odel's second incompleteness theorem.
Thus, by Theorem~\ref{thm:pi1g}, a necessary and sufficient condition for the truth of all the  G\"odelian  $\Pi_1$-sentences of $T$ is the consistency of $T$ with ${\tt Con}_T$, a condition obviously implied by $\omega$-consistency; though, this condition is stronger than the mere consistency of $T$, see  \cite[Theorem~36]{Isaacson11}.
For investigating  the truth of G\"odelian $\Pi_{n+1}$-sentences (and $\Sigma_{n+1}$-sentences) we make a definition and an observation. Before that let us note that no G\"odelian $\Sigma_1$-sentence of a consistent extension of  $\textit{\textbf{Q}}$ can be true:

\begin{proposition}[On the Truth of G\"odelian  $\Sigma_1$-Sentences]\label{prop:tg}
No G\"odelian $\Sigma_1$-sentence of $T$ can be true if $T$ is consistent.
\end{proposition}
\begin{proof}
If a G\"odelian $\Sigma_1$-sentence of $T$ were true, then it would have been provable in $\textit{\textbf{Q}}$, and this would have contradicted  Lemma~\ref{lem:ung} for  consistent $T$.
\end{proof}

\begin{definition}[$\Gamma$-Soundness]\label{def:sound}
Let $\Gamma$ be a class of sentences. A theory $S$ is called $\Gamma$-sound when every $S$-provable $\Gamma$-sentence is true.
\hfill $\diamond$
\end{definition}

The following lemma has been proved for $\Gamma\!\boldsymbol=\!\Sigma_1,\Sigma_2$ in   \cite[Theorems~25, 27, 30,32]{Isaacson11}:

\begin{lemma}[On Extensions of $\Gamma$-Sound Theories]\label{lem:2s}
Let $\Gamma$ be a class of sentences that is closed under disjunction.
If $T$ is a $\Gamma$-sound theory, then for every sentence $\varphi$,  either $T\!+\!\varphi$  or $T\!+\!\neg\varphi$ is $\Gamma$-sound.
\end{lemma}
\begin{proof}
If neither $T\!+\!\varphi$ nor $T\!+\!\neg\varphi$ is $\Gamma$-sound, then for some false $\Gamma$-sentences $\varsigma$ and $\varsigma'$ we have $T\!+\!\varphi\vdash\varsigma$ and $T\!+\!\neg\varphi\vdash\varsigma'$. Thus, $T\vdash\varsigma\!\vee\!\varsigma'$, and $\varsigma\!\vee\!\varsigma'$ is a false $\Gamma$-sentence; a contradiction.
\end{proof}

One of our main results is the following necessary and sufficient condition for the truth of G\"odelian ($\Pi_{n+1}$ and $\Sigma_{n+1}$) sentences:

\begin{theorem}[The Truth of G\"odel Sentences]\label{thm:m1} Let  $n\!\geqslant\!1$.

\noindent
All the G\"odelian $\Pi_{n+1}$-sentences of $T$ are true if and only if $T$ is $\Pi_{n+1}$-sound.

\noindent
All the G\"odelian $\Sigma_{n+1}$-sentences of $T$ are true if and only if $T$ is $\Sigma_{n+1}$-sound.
\end{theorem}
\begin{proof}
Let $\Gamma$  be any of $\Pi_{n+1}$ or $\Sigma_{n+1}$.

First, suppose that $T$ is $\Gamma$-sound, and let $\gamma$ be a G\"odelian $\Gamma$-sentence of the theory  $T$. By Lemma~\ref{lem:ung} and Convention~\ref{conv} we have $\mathbb{N}\vDash\neg{\tt Pr}_T(\#\gamma)$, and so ${\tt Pr}_T(\#\gamma)$ is a false $\Sigma_1$-sentence. Now,  $T\!+\!\neg\gamma\vdash{\tt Pr}_T(\#\gamma)$, and so $T\!+\!\neg\gamma$ is not $\Sigma_1$-sound; whence, it is not $\Gamma$-sound either. Thus, by Lemma~\ref{lem:2s}, the theory $T\!+\!\gamma$ should be $\Gamma$-sound. Therefore, $\gamma$ must be true.

Now, suppose that all the G\"odelian  $\Gamma$-sentences of $T$ are true. We show that the theory $T$ is $\Gamma$-sound. Assume that $T\vdash\varsigma$ for a $\Gamma$-sentence $\varsigma$. We prove that $\varsigma$ is true. By Lemma~\ref{lem:d} there exists a $\Gamma$-sentence $\zeta$ such that $\textit{\textbf{Q}}\vdash\zeta\!\leftrightarrow\!
[\varsigma\!\wedge\!\neg{\tt Pr}_T(\#\zeta)]$. Thus, from $T\vdash\varsigma$ we have $T\vdash\zeta\!\leftrightarrow\!
\neg{\tt Pr}_T(\#\zeta)$, and so $\zeta$ is a G\"odelian  $\Gamma$-sentence of $T$. Whence, $\zeta$ is true, and so, by the soundness of $\textit{\textbf{Q}}$, we have $\mathbb{N}\vDash\varsigma$.
\end{proof}

Whence, all the G\"odelian sentences of a theory are true if and only if the theory is sound; cf. also \cite[Theorem~24.7.]{Smith13}.

\begin{remark}[On the Hierarchy of $\Pi_n,\Sigma_n$-Soundness]\label{rem:s}{\rm
Let us note that
an extension of $\textit{\textbf{Q}}$ is consistent if and only if it is $\Pi_1$-sound: indeed,
no consistent extension  of  $\textit{\textbf{Q}}$ can  prove a  false $\Pi_1$-sentence, since  the negation of such a sentence would be a true $\Sigma_1$-sentence  and so would be provable in $\textit{\textbf{Q}}$.

 One can also show that
a theory is $\Sigma_n$-sound if and only if it is $\Pi_{n+1}$-sound:
if the theory $S$ is $\Sigma_n$-sound and $S\vdash\pi$, where $\pi$ is a $\Pi_{n+1}$-sentence, then write $\pi\!=\!\forall x\,\sigma(x)$ for a $\Sigma_n$-formula $\sigma$; since for every $k\!\in\!\mathbb{N}$ we have $S\vdash\sigma(\overline{k})$, and $\sigma(\overline{k})$ is a $\Sigma_n$-sentence, then $\mathbb{N}\vDash\sigma(\overline{k})$ for every $k\!\in\!\mathbb{N}$, so  $\mathbb{N}\vDash\forall x\,\sigma(x)\!=\!\pi$.

The hierarchy of $\Sigma_n$-sound theories is strict, since  there are some  $\Sigma_n$-sound theories which are not $\Sigma_{n+1}$-sound; this is proved in e.g.
  \cite[Theorem~2.5.]{SalSer17}  and also   \cite[Theorem~4.8.]{KK17}. Therefore, the truth of (even all) the G\"odelian $\Pi_{n+1}$-sentences (respectively,  $\Sigma_{n+1}$-sentences) of a theory does not necessarily imply the truth of its G\"odelian $\Pi_{n+2}$-sentences (respectively,  $\Sigma_{n+2}$-sentences).
\hfill $\diamond$
}\end{remark}

%

\section{Rosserian Sentences and their Truth.}

In Theorem~\ref{thm:pi1g} we saw that G\"odelian  sentences of some theories could be refutable in them (though, they are always unprovable in consistent theories, see Lemma~\ref{lem:ung}).
Rosser's trick (\cite{Rosser36}) constructs an independent sentence for a given theory (which is an {\sc re} extension of $\textit{\textbf{Q}}$), when it is consistent.
Before going into Rosser's construction, let us note that no construction similar to G\"odel's can result in an independent sentence.

\begin{definition}[Pseudo-G\"odelian Sentences]\label{def:pg}
Let us call $\psi$  a pseudo-G\"odelian sentence of a theory $T$ when there are some 
propositional formulas $C_1(p),\cdots,C_n(p)$, over the one propositional variable $p$, and there is 
one  propositional formula $B(p_1,\cdots,p_n)$, over the propositional variables $p_1,\cdots,p_n$,  such that we have 
 $T\vdash\psi\!\leftrightarrow\!
B\big({\tt Pr}_T[\#C_1(\psi)],\cdots,
{\tt Pr}_T[\#C_n(\psi)]\big)$.
\hfill $\diamond $
\end{definition}
For example, the sentences $\mathfrak{P}_T$ and $\mathfrak{R}_T$ for which we have
$$T\vdash\mathfrak{P}_T\!\leftrightarrow\!
\big[\neg{\tt Pr}_T(\#\mathfrak{P}_T)\!\wedge\!\neg{\tt Pr}_T(\#[\neg\mathfrak{P}_T])\big] \textrm{   and   } T\vdash\mathfrak{R}_T\!\leftrightarrow\!
\big[{\tt Pr}_T(\#\mathfrak{R}_T)\!\rightarrow\!\neg{\tt Pr}_T(\#[\neg\mathfrak{R}_T])\big],$$
are both some peudo-G\"odelian sentences of $T$.
For an alternative formulation of the following result see \cite[Exercise~1, p.~149]{Smorynski85}.
\begin{proposition}[Decidability of Pseudo-G\"odelian  Sentences]\label{prop:pg}

\,

\noindent
No pseudo-G\"odelian sentence of $U\!=\!T\!+\!\neg{\tt Con}_T$ can be independent from $U$.
\end{proposition}
\begin{proof}
Let $\psi$ be a pseudo-G\"odelian sentence of the theory  $U\!=\!T\!+\!\neg{\tt Con}_T$. For every sentence $\zeta$ we have $U\vdash{\tt Pr}_U(\#\zeta)$. Now,   $B(\boldsymbol\top,\cdots,\boldsymbol\top)$, where $\boldsymbol\top$ denotes the propositional truth,  is equivalent to either $\boldsymbol\top$ or $\boldsymbol\bot$, where $\boldsymbol\bot$ denotes the propositional falsum. If    $B(\boldsymbol\top,\cdots,
\boldsymbol\top)
\equiv\boldsymbol\top$, then
$U\vdash B\big({\tt Pr}_U[C_1(\#\psi)],\cdots,
{\tt Pr}_U[C_n(\#\psi)]\big)$; and if $B(\boldsymbol\top,\cdots,
\boldsymbol\top)
\equiv\boldsymbol\bot$, then  
$U\vdash \neg B\big({\tt Pr}_U[C_1(\#\psi)],\cdots,
{\tt Pr}_U[C_n(\#\psi)]\big)$.
As a result, we have either $U\vdash\psi$   or $U\vdash\neg\psi$; thus $\psi$ is not independent from $U$.
\end{proof}

\noindent
In both of the above examples, it can be seen that $U\vdash\neg\mathfrak{P}_U$ and $U\vdash\neg \mathfrak{R}_U$ hold for the theory $U=T\!+\!\neg{\tt Con}_T$. Thus, for getting independent sentences (of consistent theories) one should go beyond the (pesudo-)G\"odelian sentences.

The $T$-provability predicate ${\tt Pr}_T(x)$ in Convention~\ref{conv} is usually constructed from a $T$-{\em proof predicate} ${\tt prf}_T(y,x)$, as ${\tt Pr}_T(x)\!\boldsymbol=\!\exists y\,{\tt prf}_T(y,x)$; where ${\tt prf}_T(y,x)$ is a $\Delta_1$-formula stating that ``$y$ is (the G\"odel code of) a proof in $T$ of the formula (coded by) $x$''.

\begin{convention}\label{convr}
Let us fix a proof predicate  of $T$ as the  $\Delta_1$-formula ${\tt prf}_T(y,x)$ that satisfies the following for every sentence $\varphi$:

\qquad $T\vdash\varphi \; \iff \; \textit{\textbf{Q}}\vdash{\tt prf}_T(\overline{m},\#\varphi)$ for some $m\!\in\!\mathbb{N}$.

\qquad $T\nvdash\varphi \; \iff \;  \textit{\textbf{Q}}\vdash\neg{\tt prf}_T(\overline{n},\#\varphi)$ for every $n\!\in\!\mathbb{N}$.
\hfill $\diamond$
\end{convention}

\begin{definition}[Rosserian Provability and Rosserian Sentences]\label{def:r}
The following  $\Sigma_1$-formula, with the free variable $x$, 
is  the   Rosserian Provability predicate of $T$:
$${\tt R.Pr}_T(x) = \exists y\,[{\tt prf}_T(y,x)\wedge\forall z\!<\!y\,\neg{\tt prf}_T(z,\neg x)].$$

\noindent
A sentence $\rho$ is called a Rosserian sentence of $T$ when we have
$T\vdash\rho\!\leftrightarrow\!\neg{\tt R.Pr}_T(\#\rho)$.
\hfill $\diamond $
\end{definition}

The independence of the Rosserian sentences (from the theory in question) follows form the following basic properties of the Rosserian provability:

\begin{lemma}\label{lem:unr}
If $T$ is consistent, then for every sentence $\varphi$ we have

(1)\; $T\vdash\varphi \iff  \textit{\textbf{Q}}\vdash{\tt R.Pr}_T(\#\varphi)$.

(2)\;  $T\vdash\neg\varphi \;\Longrightarrow\;  \textit{\textbf{Q}}\vdash\neg{\tt R.Pr}_T(\#\varphi)$.
\end{lemma}
\begin{proof}
For (1) it suffices to note that for a consistent theory $T$ we have:   $T\vdash\varphi$ if and only if the $\Sigma_1$-sentence ${\tt R.Pr}_T(\#\varphi)$ is true. For (2) suppose that $T\vdash\neg\varphi$; then by Convention~\ref{convr} for some $m$ we have $\textit{\textbf{Q}}\vdash{\tt prf}_T(\overline{m},\#[\neg\varphi])$. Now, reason inside $\textit{\textbf{Q}}$: for any $y$ with ${\tt prf}_T(y,\#\varphi)$ we should have $y\!>\!\overline{m}$, since no $i\!\leqslant\!\overline{m}$ (which are $i=0,\cdots,\overline{m}$) could satisfy ${\tt prf}_T(\overline{i},\#\varphi)$ by Convention~\ref{convr}, and so for some $z\!<\!y$, which is $z\!=\!\overline{m}$, we have ${\tt prf}_T(z,\#[\neg\varphi])$. Thus, we have   $\forall y [{\tt prf}_T(y,\#\varphi)\!\rightarrow\!\exists z\!<\!y\,{\tt prf}_T(z,\#[\neg\varphi])]$ or, equivalently, $\neg{\tt R.Pr}_T(\#\varphi)$.
\end{proof}

Now, we can characterize the independent sentences of $T$   similarly to Proposition~\ref{prop:g}:

\begin{proposition}[Characterizing Independent   Sentences]\label{prop:r}
Let $\varphi$ be a sentence.

\noindent
The following are equivalent:
\begin{enumerate}\itemindent=5em
\item[(1)] $\varphi$ is independent from $T$, i.e., $T\nvdash\varphi$ and $T\nvdash\neg\varphi$;
\item[(2)] $\varphi$ is a Rosserian sentence of some consistent extension $U$ of $T$;
\end{enumerate}
and are implied by the following:
\begin{enumerate}\itemindent=5em
\item[(3)] $T+[\varphi\!\leftrightarrow\!
\neg{\tt R.Pr}_T(\#\varphi)]$ is consistent.
\end{enumerate}
\end{proposition}
\begin{proof}
First we show the equivalence of (1) and (2).

$(1\!\Rightarrow\!2)$: By Lemma~\ref{lem:d} for some sentence $\xi$ we have
$T\vdash\xi\!\leftrightarrow\![\varphi
 \!\leftrightarrow\!\neg{\tt R.Pr}_{T+\xi}(\#\varphi)]$. Let $U\!=\!T\!+\!\xi$; then $U\vdash\varphi\!\leftrightarrow\!
\neg{\tt R.Pr}_U(\#\varphi)$ which shows that $\varphi$ is a Rosserian sentence of $U$.  We show that $U$ is consistent. If not, then $T\vdash\neg\xi$.
Thus, ($\ast$) $T\vdash\neg[\varphi
 \!\leftrightarrow\!\neg{\tt R.Pr}_{U}(\#\varphi)]$. Also,  $U\vdash\varphi$ and $U\vdash\neg\varphi$, and so by Convention~\ref{convr} there are  $m,n\!\in\!\mathbb{N}$ such that $\textit{\textbf{Q}}\vdash{\tt prf}_U(\overline{m},\#\varphi)$
 and
 $\textit{\textbf{Q}}\vdash{\tt prf}_U(\overline{n},\#[\neg\varphi])$; we can assume that $m$ and $n$ are the least such numbers. 

(i) If $m\!\leqslant\!n$, then $\textit{\textbf{Q}}\vdash {\tt prf}_U(\overline{m},\#\varphi) \wedge
\forall z\!<\!\overline{m}\neg{\tt prf}_U(z,\#[\neg\varphi])$ and so
$\textit{\textbf{Q}}\vdash {\tt R.Pr}_U(\#\varphi)$, which implies by ($\ast$) that $T\vdash\varphi$; contradicting (1).

(ii) If $n\!<\!m$, then $\textit{\textbf{Q}}\vdash
\forall y[{\tt prf}_U(y,\#\varphi)
\!\rightarrow\!
y\!\geqslant\!\overline{m}
\!\rightarrow\!
\exists z\!<\!y\, {\tt prf}_U(z,\#[\neg\varphi])]$, since one can take $z\!=\!n$, and so $\textit{\textbf{Q}}\vdash\neg{\tt R.Pr}_U(\#\varphi)$, which implies by ($\ast$) that $T\vdash\neg\varphi$; contradicting (1) again.

So, $U$ must be consistent.

$(2\!\Rightarrow\!1)$:
It suffices to show that $\varphi$ is independent from $U$. If  $U\vdash\rho$, then we should have on the one hand  $U\vdash{\tt R.Pr}_U(\#\rho)$ by Lemma~\ref{lem:unr}(1), and on the other hand   $U\vdash\neg{\tt R.Pr}_U(\#\rho)$ by Definition~\ref{def:r}; thus, $U$ could not be consistent. Also $U\vdash\neg\rho$ would imply on the one hand   $U\vdash\neg{\tt R.Pr}_U(\#\rho)$ by  Lemma~\ref{lem:unr}(2), and on the other hand   $U\vdash{\tt R.Pr}_U(\#\rho)$ by Definition~\ref{def:r}; contradicting $U$'s consistency again.

Now, we show that (3) implies (1); and so (2) too.

$(3\!\Rightarrow\!1)$: Note that the theory $T$ is consistent by the assumption. If $T\vdash\varphi$, then by Lemma~\ref{lem:unr}(1) we would have $T\vdash{\tt R.Pr}_T(\#\varphi)$, and so $T\vdash\neg[\varphi\!\leftrightarrow\!
\neg{\tt R.Pr}_T(\#\varphi)]$. If $T\vdash\neg\varphi$, then Lemma~\ref{lem:unr}(2) would imply that $T\vdash\neg{\tt R.Pr}_T(\#\varphi)$, and so we would have $T\vdash\neg[\varphi\!\leftrightarrow\!
\neg{\tt R.Pr}_T(\#\varphi)]$ again.
\end{proof}

\begin{remark}[L\"ob-Like Rule for Rosserian  Provability]\label{rem:rlr}{\rm
Let us note that the contraposition of the implication $(1\!\Rightarrow\!3)$ in Proposition~\ref{prop:r} says that if $T\vdash\varphi\!\leftrightarrow\!{\tt R.Pr}_T(\#\varphi)$, i.e., if $\varphi$ is a Rosser-type Henkin sentence so called by Kurahashi  (in~\cite{Kurahashi16}), then $\varphi$ is not independent from $T$. Actually, it is shown in~\cite{Kurahashi16} that there are {\em standard} proof predicates ${\tt prf'}_T(y,x)$ which have independent Rosser-type Henkin sentences, and there are standard proof predicates ${\tt prf''}_T(y,x)$ none of whose Rosser-type Henkin sentences are independent.
The latter proof predicates satisfy $(1\!\Rightarrow\!3)$ in Proposition~\ref{prop:r} and satisfy a L\"ob-like rule for Rosserian provability; while the former ones do not satisfy $(1\!\Rightarrow\!3)$ in Proposition~\ref{prop:r}, and  do not satisfy any L\"ob-like rule for Rosserian provability. So, the implication $(1\!\Rightarrow\!3)$ in Proposition~\ref{prop:r} depends on the proof predicate ${\tt prf}_T(y,x)$, and is not robust.
\hfill $\diamond$
}\end{remark}

Unlike G\"odelian $\Pi_1$-sentences, all the Rosserian $\Pi_1$-sentences of consistent theories are true, and like G\"odelian $\Sigma_1$-sentence, all of their Rosserian $\Sigma_1$-sentences are false:

\begin{theorem}[On the Truth of the Rosserian $\Pi_1,\Sigma_1$-Sentences]\label{thm:1r}
Every Rosserian $\Pi_1$-sentence of $T$ is true, and every  Rosserian $\Sigma_1$-sentence of $T$ is false, if $T$ is consistent.
\end{theorem}
\begin{proof}
If a Rosserian $\Pi_1$-sentence were false, then its negation would be  a true $\Sigma_1$-sentence, and so would be provable in the theory $\textit{\textbf{Q}}$; contradicting Rosser's theorem on the independence of Rosserian sentences (Proposition~\ref{prop:r}). If a Rosserian $\Sigma_1$-sentence were true, then it would be provable in the theory $\textit{\textbf{Q}}$; contradicting the unprovability of Rosserian sentences.
\end{proof}

However, for $n\!\geqslant\!1$, the truth of all  the G\"odelian $\Pi_{n+1}$-sentences is equivalent to the truth of all the Rosserian $\Pi_{n+1}$-sentences; and the truth of all  the G\"odelian $\Sigma_{n+1}$-sentences is equivalent to the truth of all the Rosserian $\Sigma_{n+1}$-sentences:

\begin{theorem}[On the Truth of the Rosserian $\Pi_{n+1},\Sigma_{n+1}$-Sentences]\label{thm:allr}
Fix $n\!\geqslant\!1$.

\noindent
All the Rosserian $\Pi_{n+1}$-sentences of $T$ are true if and only if $T$ is $\Pi_{n+1}$-sound.

\noindent
All the Rosserian $\Sigma_{n+1}$-sentences of $T$ are true if and only if $T$ is $\Sigma_{n+1}$-sound.
 \end{theorem}
\begin{proof}
This can be proved in the same lines of the proof of  Theorem~\ref{thm:m1}   by using Lemma~\ref{lem:unr} (instead of Lemma~\ref{lem:ung}) which implies that
${\tt R.Pr}_T(\#\rho)$ is a false $\Sigma_1$-sentence  when $\rho$ is a Rosserian sentence  of  consistent theory $T$.
\end{proof}

Whence, all the Rosserian sentences of $T$ are true if and only if $T$ is sound; cf. \cite[Theorem~24.7.]{Smith13}.

\section{Conclusions.}
The first one who talked about the truth of  G\"odelian sentences was G\"odel himself \cite{Godel31}. This turned into a serious debate with \cite{Godel51} in which (what we call now) the {\em G\"odel Disjunction} was announced; see \cite{Feferman06} and \cite{HW16} and the references therein.
The so called {\em Anti-Mechanism Thesis}, or the {\em Lucas-Penrose Argument}, started with \cite{Lucas61} and popularized by \cite{Penrose89}; see also \cite{NN58} and \cite{Putnam60}. After that, there has been a large discussion on the truth of G\"odelian sentences; see e.g.
 \cite{Dummett63}, \cite{Smorynski77}, \cite{Boolos90}, \cite{Shapiro98},
\cite{Tennant01}, \cite{Raatikainen05}, \cite{Peregin07}, \cite{Milne07}, \cite{Sereny11}, \cite{Isaacson11}, \cite{BS12},
\cite{PP15} and  \cite{PP16}.
As shown above, the consistency of a theory need not imply the truth of (all of) its G\"odelian ($\Pi_1$-)sentences; but does imply the truth of its all Rosserian  $\Pi_1$-sentences.
One wonders why the proponents of the anti-mechanism thesis have not used the Rosserian  ($\Pi_1$-)sentences for their reasoning; since the truth of those sentences are straightforward (and immediately follows from the consistency of the theory). Though, the opponents have argued that actually for ``seeing'' the truth of G\"odelian  ($\Pi_1$-)sentences one should ``see'' (at least) the consistency of the theory (and indeed, more than that).  Our old and new
results are summarized in the following diagram; note that the conditions get (strictly) stronger from bottom to top.

\begin{table}[h]
\begin{center}
{\small
\begin{tabular}{||rcl||}
\hline
Soundness &
$\!\!\!\equiv\!\!\!$ &
Truth of G\"odelian and  Rosserian  Sentences \\
 & $\!\!\vdots\!\!$ & \\
$\Sigma_{n+1}$ ($\Pi_{n+2}$) Soundness & $\!\!\!\equiv\!\!\!$ &
Truth of G\"odelian, Rosserian $\Sigma_{n+1},\Pi_{n+2}$ Sentences \\
 & $\!\!\vdots\!\!$ & \\
$\Sigma_2$ ($\Pi_3$) Soundness   & $\!\!\!\equiv\!\!\!$ &
Truth of G\"odelian,  Rosserian $\Sigma_2,\Pi_3$ Sentences \\
$\Sigma_1$ ($\Pi_2$) Soundness  & $\!\!\!\equiv\!\!\!$ &
Truth of G\"odelian, Rosserian $\Pi_2$ Sentences \\
Consistency of $T\!\boldsymbol+\!{\tt Con}_T$ & $\!\!\!\equiv\!\!\!$ &
Truth of  G\"odelian $\Pi_1$ Sentences \\
Consistency ($\Pi_1$ Soundness) & $\!\!\!\equiv\!\!\!$ &
Truth of  Rosserian $\Pi_1$ Sentences \\[-0.5em]
\dotfill & $\!\!\!\!\!\ldots\!\!\!\!\!$  & \dotfill \\
Consistency ($\Pi_1$ Soundness) & $\!\!\!\equiv\!\!\!$ &
Falsity of G\"odelian, Rosserian $\Sigma_1$ Sentences \\
\hline
\end{tabular}
}
\end{center}
\end{table}

\paragraph{Acknowledgements.}
{This research is supported by the grant $\mathcal{N}^{\underline{O}}$\,{\sf 98013437} of the Iran National Science Foundation ($\mathtt{INSF}$).} The authors warmly thank {\sc Kaave Lajevardi}  for the most helpful discussions and comments; this is a continuation of a project that he started a while ago and which have resulted in \cite{LajSal19} and some other works to appear in the near future.

\end{document}